%% file: p-adic_heisenberg.tex
\documentclass[12pt]{amsart}
\usepackage{amssymb,latexsym,amsmath,amsthm,amscd}
\usepackage{setspace}
\usepackage[all]{xy} \xyoption{arc}
\usepackage[left=3cm,top=2cm,right=3cm,bottom = 2cm]{geometry}
\usepackage{graphicx}
\usepackage[usenames,dvipsnames]{xcolor}
\usepackage{mathtools}
\usepackage{charter}

\theoremstyle{plain}
\newtheorem{thm}{Theorem}[section]
\newtheorem{lem}[thm]{Lemma}
\newtheorem{prop}[thm]{Proposition}
\newtheorem{cor}[thm]{Corollary}

\theoremstyle{definition}
\newtheorem{dfn}[thm]{Definition}

\theoremstyle{remark}
\newtheorem{rmk}[thm]{Remark}

\input{preamble}

\DeclareMathOperator{\alg}{alg}

\newcommand{\bone}{\mathbf{1}}
\DeclareRobustCommand{\stirling}{\genfrac\{\}{0pt}{}}

\begin{document}
\title{On $p$-adic modularity in the $p$-adic Heisenberg algebra}
\author{Cameron Franc and Geoffrey Mason}
\date{}

\begin{abstract}
  We establish existence theorems for the image of the normalized character map of the $p$-adic Heisenberg algebra $S$ taking values in the algebra of Serre $p$-adic modular forms $M_p$. In particular, we describe the construction of an analytic family of states in $S$ whose character values are the well-known $\Lambda$-adic family of $p$-adic Eisenstein series of level one built from classical Eisenstein series. This extends previous work treating a specialization at weight $2$, and illustrates that the image of the character map contains nonzero $p$-adic modular forms of every $p$-adic weight. In a different direction, we prove that for $p=2$ the image of the rescaled character map contains every overconvergent $2$-adic modular form of weight zero and tame level one; in particular, it contains the polynomial algebra $\QQ_2[j^{-1}]$. For general primes $p$, we study the square-bracket formalism for $S$ and develop the idea that although  states in $S$ do not generally have a conformal weight, they can acquire a $p$-adic weight in the sense of Serre. 
\end{abstract}
\maketitle
\tableofcontents

\section{Introduction}
The occurrence of \emph{modularity} in the theory of vertex operator algebras (VOAs) is by now a well-known, even commonplace, phenomenon.\ `Modularity' usually refers to elliptic modular forms although many other types of modular and automorphic objects intervene in the theory.\ One runs across elliptic functions, Siegel modular forms, quasimodular forms, Jacobi  forms and mock modular forms for example, not to mention vector-valued versions of these.\ Absent from this list, however, are $p$-adic modular forms and their variants.

The last several decades have witnessed an explosion of work in number theory using $p$-adic methods to solve outstanding classical problems, as well as introducing new $p$-adic versions of old results and conjectures.\  Of relevance to the present paper is the theory of $p$-adic modular forms as introduced by Serre and Katz \cite{Katz,Serre2}.\ Among other applications, these new modular forms are frequently used to define $p$-adic variants of $L$-functions via the theory of $p$-adic interpolation. See for example \cite{Hida} for an overview of this theory, and \cite{Vonk} for a concrete discussion of the theory of overconvergent modular forms. Some of the essential facts are also reviewed in Section \ref{SN} below.

Given the close connection between VOAs and modular forms, it is natural to hypothesize a $p$-adic theory of VOAs that would extend this connection to $p$-adic modular forms. In \cite{FMpadic} we introduced such a theory  by adopting a set of axioms that naturally arises by completing the usual axioms of a standard VOA (or, as we shall call them, \emph{algebraic} VOAs) with respect to some $p$-adic norm.\ This is very different from simply tensoring with the field $\QQ_p$ of $p$-adic numbers, a situation that is mirrored in Serre's theory \cite{Serre}.\ Indeed, Serre notes that merely tensoring with $\QQ_p$ essentially produces nothing new since the space of modular forms of level one has a basis defined over the integers.\ Similarly, although the resulting completed $p$-adic VOAs of \cite{FMpadic} have many properties akin to those of an algebraic VOA, we emphasize that a $p$-adic VOA is \emph{not} a VOA in the usual sense:\ for example, one axiom of algebraic VOAs (the truncation property of vertex operators) is that certain generating series (fields applied to states) are finite-tailed Laurent series.\ By contrast, in our $p$-adic theory these series can have essential singularities, although the coefficients of the polar terms tend to zero in the underlying $p$-adic topology. While novel in the algebraic theory of VOAs, such series are commonplace in $p$-adic analysis and geometry.

Heisenberg algebras are central objects in both mathematics and physics, providing one of the most basic and amenable examples of relevance. Use of the word `algebra' in this context is a shorthand for either `Heisenberg Lie algebra' or `Heisenberg vertex operator algebra'. It is convenient to  denote the algebraic Heisenberg VOA of rank $1$ over $\QQ_p$ by $S_{\alg}$.\ This is to distinguish it from its $p$-adic completion $S$, called the $p$-adic Heisenberg algebra that was constructed in \cite{FMpadic} and shown to be a $p$-adic VOA.\ They both occur in the following commuting diagram that we will shortly explain and which goes to the heart of the main results of the present paper.
\begin{eqnarray}\label{Sdiag}
\begin{xymatrix}
{
& S\ar[rrr]^{\hat{f}}&&& M_p   \\
& S_{alg}\ar[rrr]_{f}\ar[u]&&& \QQ_p[E_2, E_4, E_6]\ar[u]
 }
\end{xymatrix}
\end{eqnarray}
The left vertical map is the natural containment of  $S_{\alg}$ in its $p$-adic completion $S$.\ $M_p$ is the space of (Serre) $p$-adic modular forms of tame level $1$. It is the $p$-adic completion of $\QQ_p[E_4, E_6]$, however, $E_2$ also lies in $M_p$ and the right vertical map  is the natural containment.\ $\QQ_p[E_2, E_4, E_6]$ is the ring of quasimodular forms of level $1$.\

As for the horizontal maps, $f$ is a renormalized version of the usual character map ($1$-point function)
of algebraic VOAs that associates a formal $q$-series (a certain graded trace) to a state in the VOA.\ This is an important, though technical, aspect of the general theory and we say more about it in Subsection \ref{SScharmap}.\ It is known \cite{DMNQuasi, MT2, MT}, that $f$ maps into the ring of quasimodular forms and in fact the last two references prove that $f$ is a \emph{surjection}.\ Furthermore it is proved in \cite{FMpadic} that $f$ is $p$-adically continuous, so that it extends to the completions and this is the upper map
$\hat{f}$ which is a continuous linear map of $p$-adic Banach spaces.

This diagram also helps explain our emphasis in the present paper on the Heisenberg algebra.\ In principle we would like to consider similar diagrams in which  $S_{\alg}$ is replaced with other algebraic VOAs  and $S$ with their $p$-adic completions.\ However, we presently do not generally have a good understanding of either the vertex operator structure of the completion or of the image of $f$ and this greatly complicates any study of $\hat{f}$, for example a description of its image.

As it is, since Diagram \eqref{Sdiag} commutes and $f$ is surjective (see below) then $\im\hat{f}$ contains $\QQ_p[E_2, E_4, E_6]$  and we can ask for a description of the precise image of this map.\ The most natural expectation is that $\hat{f}$ \emph{surjects onto $M_p$}.\ This conclusion --- if true --- remains open, and the main purpose of the present paper is to develop techniques and some explicit results that contribute towards its affirmative resolution.\

We already proved in \cite{FMpadic} that the image of $\hat{f}$ is \emph{strictly larger} than the space of quasimodular forms, i.e., the image  contains new $p$-adic modular forms, and additional examples of a similar nature are developed in \cite{BarakeFranc}.\ The main arithmetic results of the present paper are encapsulated in

\begin{thm}\label{t:main1}
The following hold:
\begin{enumerate}
    \item For every prime $p$ and for every even weight $k$ in $p$-adic weight space $X$, $\im\hat{f}$ contains
    the $p$-adic Eisenstein series $G_k^*$ of weight $k$. In fact, there exists a $\Lambda$-adic family of continuously varying states in the $p$-adic Heisenberg algebra that lifts this family of Eisenstein series.
    \item When $p=2$, $\im\hat{f}$ contains the space of $2$-adic, weight $0$, overconvergent modular forms $M_0^{\dagger}(7/4)$.\ In particular, 
    $\im\hat{f}$ contains $\QQ_2(j^{-1})$.
\end{enumerate}
\end{thm}
\noindent
For the benefit of readers uninured to the language of $p$-adic and $\Lambda$-adic modular forms, the relevant terms are explained in Section \ref{SN} and \eqref{M0dagger}.

The reason why we do not obtain surjectivity in weight zero when $p=2$ is that the states in $S$ that we use as preimages of the powers $j^{-n}$  do not have sufficiently nice asymptotic behaviour as $n$ varies.\ See Corollary \ref{c:main2adic} and the surrounding discussion for more details on this point.\ The reason why we restrict to the prime $p=2$ is to simplify computations, as in this case the space of $2$-adic modular forms of tame level one has a particularly simple description as the Tate algebra $\QQ_2\langle j^{-1}\rangle$ of series in $j^{-1}$ with coefficients that tend to zero $2$-adically \cite{Vonk}.\ See the start of Section \ref{s:weightzero} for a discussion of this point.\ A very similar description of weight zero forms holds when $p=3$ and $p=5$ so that in principle one might be able to extend our computations to these primes. However, it seems that new ideas would be required for a significantly more general result, and so for this reason we have contented ourselves here with the restrictions and results stated in Theorem \ref{t:main1}.

We now discuss some additional aspects of the proof of the Theorem.\ This may be done with the aid of a diagram similar to (\ref{Sdiag}), namely
\begin{eqnarray}\label{ugdiag}
\begin{xymatrix}
{
& u\ar[rrr]^{\hat{f}}&&& g   \\
& (u_i)\ar[rrr]_{f}\ar[u]&&& (g_i)\ar[u]
 }
\end{xymatrix}
\end{eqnarray}

Here, $g$ is a designated $p$-adic modular form that we want to show lies in $\im\hat{f}$; $g$ is the $p$-adic limit of a Cauchy sequence of classical modular forms $g_i$, say of weight $k_i$.\ So the right vertical arrow is just notation.\ We want to pull this back to the Heisenberg VOA.\ Because $f$ is a \emph{graded} surjection (again, see below) there are states $u_i\in(S_{\alg})_{[k_i]}$ mapping onto $g_i$.\ Because $f$ has a very large kernel there is generally not a unique choice for $u_i$ and in any case the resulting sequence of states $(u_i)$ may not be a Cauchy sequence in $S_{\alg}$.\ As long as it is, however, and if its limit $u$ is contained in $S$, then the $p$-adic continuity of $\hat{f}$ ensures that $f(u)=g$ as required.\ The $p$-adic modular forms intervening in the statement of Theorem \ref{t:main1} are emblematic of cases for which we can make this procedure work.

We turn to a discussion of weights in $S$ and $S_{\alg}$ and first note that when the situation of \eqref{ugdiag} prevails, Serre \cite{Serre2} tells us that $(k_i)$ is a Cauchy sequence
having a limit, say $k$, in $p$-adic weight space $X$.\ This is the $p$-adic weight of $g$ and it is natural to say that $u$ also has a weight $k$ in order to make $\hat{f}$ a weight-preserving map.\ Actually, we say that $u$ has \emph{$X$-weight} $k$ so as to distinguish it from the other types of weights occurring in a VOA (cf.\ the next paragraph).\ This notion of $X$-weight in $S$ generalizes the usual square bracket weight of a state in $S_{\alg}$.\  The latter type of weight belongs to $\ZZ$, which is embedded in $X$ as a dense subspace.\ So we see that not only Fock space itself, but also the modular forms and the $X$-weights of states all arise from the process of  $p$-adic completion.

Weights of states in a VOA such as $S_{\alg}$ are usually described as eigenvalues of some specific operators,
and they are integers.\ The notion of the $X$-weight of a limit state in $S$ introduced above is quite different.\ $X$-weights lie in the $1$-dimensional $p$-adic Lie group $X$ and they are not eigenvalues of anything.\ In order to reconcile these various notions of weights we recall some details about weights in an algebraic VOA.

VOA theorists will be familiar with the fact 
that in $S_{\alg}$, indeed in any VOA $V$,  there are two special semisimple operators $L(0)$ and $L[0]$
with point spectra in $\ZZ$ such that 
\[V=\oplus_n V_n=\oplus_n V_{[n]}\]
and $V_n$, $V_{[n]}$ are the eigenspaces for $L(0)$ and $L[0]$, respectively, with eigenvalue $n$. Further details are given in Section \ref{Ssqbrack}. The grading on $V$ conferred by $L(0)$ is usually called the \emph{conformal grading}, and if $v\in V_n$ we say that $v$ has (conformal) weight $n$.\ The second grading has, for some reason, not (yet) acquired an official name.\ We will call it the \emph{square bracket grading} and say that $v\in V_{[n]}$ has \emph{square bracket weight $n$}.\ The importance of the square bracket grading arises from the fact that, with respect to it,  $f$ in \eqref{Sdiag} becomes a \emph{graded map}, a fact that we already alluded to above.\ Thus
if $u\in (S_{\alg})_{[k]}$ then $f(u)$ is a quasimodular form of weight $k$. 

Unlike their behaviour in $S_{\alg}$, and as far as the $p$-adic VOA $S$ is concerned, the spectral properties of $L(0)$ and $L[0]$ diverge significantly.\ $S$ does \emph{not} carry a conformal grading extending that on $S_{alg}$.\ On the contrary it is shown in \cite[Proposition 7.3]{FMpadic} that all $L(0)$-eigenstates in $S$ are already contained in $S_{\alg}$.\ On the other hand, $L[0]$ is described by an infinite sum \eqref{0mode} which does not converge on $S$. To circumvent this apparent difficulty we utilize a family of $p$-adic Banach spaces introduced in \cite{FMpadic} and denoted by $S_r$ for $r\in\RR$.\ They are defined by imposing growth conditions on power series coefficients, similar to the types of growth conditions that arise in the work of Katz \cite{Katz} and the theory of overconvergent modular forms which dates back to work of Dwork, Coleman and many others --- see \cite{Vonk} for an accessible survey of some of this theory.\ Moreover, these spaces $S_r$ are nested
\[S_{\alg}\subseteq S_{r_1}\subseteq S_{r_2}\subseteq S=S_1\] whenever $r_1\geq r_2\geq 1$.\ The spaces $S_r$ have a norm $\abs{\cdot}_r$ that is \emph{stronger} than the supremum norm $\abs{\cdot}_1$ giving rise to $S$ itself.\ We show that $L[0]$ operates continuously on $S_r$ for $r\geq p^{1/p}$ (Theorem \ref{thmL}) and we can study the point spectrum for this action.\ It turns out that the $p$-adic states
$u$ implicitly intervening in  Theorem \ref{t:main1} \emph{do} lie in an $S_r$ for large enough $r$ and there they not only have an $X$-weight but are also eigenstates for $L[0]$.\ We show in Theorem \ref{thmL0eigen} that \emph{every} $p$-adic integer occurs as an eigenvalue of $L[0]$.\ 
One expects that, in stark contrast to its action on $S_{\alg}$,  $L[0]$ typically has infinite-dimensional weight spaces and we show in Corollary \ref{corzero} that, at least when $p=2$, the $0$-weight space is indeed infinite-dimensional.

We leave open the problem of establishing a more direct connection between the submodules $S_r$ and the theory of overconvergent modular forms.\ Such a connection could follow from a reformulation of the $p$-adic axioms introduced in \cite{FMpadic} using instead analogs of the so-called genus zero axioms for algebraic VOAs, as discussed in \cite{FHL} and emphasized in \cite{FBZ}.\ In the $p$-adic theory, the algebro-geometric foundations would naturally be replaced by the theory of rigid analytic geometry \cite{BGR}.\ Likewise, Serre's description of $p$-adic modular forms would need to be replaced by Katz's perspective of $p$-adic modular forms as sections of bundles living over the ordinary parts of classical modular curves.\ Such an intrinsic and geometric description of the foundations of $p$-adic VOAs could lead to the introduction of less explicit, and therefore potentially more general, techniques for studying $p$-adic characters of $p$-adic VOAs.

The paper is organized as follows:\ in Section \ref{SN} we explain some notations that we use throughout and provide additional background on a range of topics.\ These include $p$-adic weight space; special numbers such as Stirling and Bernoulli-type numbers, which are  ubiquitous thanks to the nature of Zhu's exponential change-of-variables formula;
modular forms, both classical and $p$-adic; $\Lambda$-rings; Heisenberg algebras.\ In Section \ref{Ssqbrack} we cover the square bracket formalism, explain $X$-weights, determine in Theorem \ref{thmL} when $L[0]$ is $p$-adically continuous,
and construct the Cauchy sequences $(u_i)$ of Diagram \eqref{ugdiag} that lead to the $p$-adic Eisenstein series in Theorem \ref{t:main1}(1).\ The proof of this part of the Theorem is completed in Section \ref{SLambda}.\ In the more elaborate Sections \ref{s:weightzero} and \ref{Sp=2I} we prove  
Theorem \ref{t:main1}(2) and show in Corollary \ref{corzero} that when $p=2$ the $L[0]$-eigenspace for eigenvalue $0$ is infinite-dimensional.\ In the final Section \ref{Shn} we treat the actions on $S_r$ of the square bracket modes $h[n]$ of the weight $1$ Heisenberg state $h$.\ This is not needed for the proof of Theorem \ref{t:main1} but is included here both because of the similarity to previous calculations and because we anticipate that it will be helpful in identifying new $p$-adic VOAs in the future.

\section{Background and notation}\label{SN}
Fix a prime $p$.\
For an integer $n$ with $n=p^km$ and 
$\gcd(m, p)=1$ we write $\nu(n)\df k$ and $\abs{n}\df p^{-k}$. These are the $p$-adic \emph{valuations} and \emph{absolute values}, respectively. We omit a subscript of $p$ on these notations in order to avoid a profusion of subscripts.
\subsection{Weight space}\label{SSXwt}
 Let $\ZZ_p$ and $\QQ_p$ denote the ring of $p$-adic integers and its quotient field of $p$-adic numbers, respectively.\ Following Serre \cite{Serre2}, we denote \emph{$p$-adic weight space} as 
 \begin{align}\label{wtsp}
& X = \varprojlim_m \ZZ/p^m(p-1)\ZZ \cong \ZZ_p\times \ZZ/(p-1)\ZZ,
\end{align} 
so that when $p=2$ we have simply $X = \ZZ_2$.\ The space $X$ is a  one-dimensional $p$-adic Lie group that contains $\ZZ$ as a dense subgroup embedded diagonally. Points $(x,y)$ and $(u,v)$ are close in weight space if $x \equiv u \pmod{p^N}$ for some $N\gg 0$, and $y\equiv v \pmod{p-1}$. In particular, integers $a$ and $b$ are close in weight space if $a \equiv b \pmod{\phi(p^{N+1})}$ for some $N\gg 0$, where $\phi$ denotes Euler's totient function satisfying $\phi(p^{N+1}) = (p-1)p^N$.

We will also regard elements of $X$ as $p$-adic characters of the unit group $\ZZ_p^\times$.\ This latter group occurs as the middle term of a (split) short exact sequence
\begin{align}\label{ses1}
 & 1\rightarrow 1+p\ZZ_p\rightarrow \ZZ_p^{\times}\rightarrow \ZZ/(p-1)\ZZ\rightarrow 1 .  
\end{align}
We set $V_p \df \Hom(\ZZ_p^\times,\ZZ_p^\times)$, the set of continuous endomorphisms of $\ZZ_p^\times$ equipped with the topology of uniform convergence.\ Then there is a natural continuous homomorphism
\[
  \veps \colon X \to V_p
\]
that extends the natural map on $\ZZ$ taking an integer $n$ to the endomorphism $v\mapsto v^n$ for $v\in\ZZ_p^{\times}$. More generally if $k \in X$, to describe how $\veps(k)$ acts on $v$ we write $k = (s,u)$ according to the direct product decomposition \eqref{wtsp} and decompose $v = v_1v_2$ where $v_1^{p-1} = 1$ and $v_2 \equiv 1 \pmod{p}$ (cf. \eqref{ses1}). Then:
\[
  v^k \df \veps(k)(v) =  v_1^kv_2^k = v_1^uv_2^s.
\]
The map $\veps$ is injective if $p=2$ and bijective if $p$ is odd.

The \emph{even} weights are the elements of $2X$, equivalently, those $k \in X$ with $(-1)^k=1$.\ For odd $p$ this is equivalent to the component $u \in \ZZ/(p-1)\ZZ$ of $k$ being even. When $p =2$ the even weights are the elements of $2\ZZ_2$.\ Naturally, an \emph{odd} weight is one that is not even, and one sees that $k$ is even if, and only if, $1+k$ is odd.

\subsection{Special numbers}\label{SSspecnos}
The following sequences will play an important r\^{o}le in our computations, just as they do in the theory of $p$-adic $L$-functions, cf.\ \cite{Washington}. \emph{Bernoulli numbers} are defined by the series
\[
\sum_{k\geq0} \frac{B_k}{k!}z^k\df \frac{z}{e^z-1},   
\]
and more generally \emph{generalized Bernoulli polynomials} are defined by
\[
    \sum_{k\geq 0} B_k^{(\ell)}(x)\frac{z^k}{k!}\df e^{zx}\left(\frac{z}{e^z-1}\right)^{\ell}.
\]
\emph{Stirling numbers of the first kind} $s(n, k)$ are coefficients of the falling factorial
\[
\sum_{k=0}^n s(n, k)z^k \df z(z-1)...(z-n+1).
\]
\emph{Stirling numbers of the second kind} $\stirling{n}{k}$ are defined for nonnegative $k$ by the series
\[
  \sum_{n\geq k} \stirling{n}{k}  \frac{z^n}{n!} \df \frac{1}{k!}(e^z-1)^k.
\]
\subsection{Elliptic modular forms}\label{SSEis}
We use several different normalizations of classical level $1$ Eisenstein series.\ For \emph{even} $k\geq 2$ we set
\[
 E_k \df -\frac{B_k}{k!}+\frac{2}{(k-1)!}\sum_{n\geq 1}\sigma_{k-1}(n)q^n
\]
while we convene that $E_k=0$ for odd $k$.\ We also introduce, for positive integers $r$, $s$, 
\begin{align}\label{wideE}
 \widehat{E}_{r+s} &\df (-1)^{r+1}r\binom{r+s-1}{s} E_{r+s}.
\end{align}
This is \emph{symmetric} in $r$ and $s$, however care is needed when manipulating these normalizations.\ Although $\widehat{E}_{r+s}$ is a scalar multiple of $E_{r+s}$ the scalar in question depends on $\{r, s\}$.\ These normalizations occur naturally in character formulas for VOAs, for example in \eqref{Zform}.
\[
  G_k \df \frac{(k-1)!}{2}E_k = -\frac{B_k}{2k}+\sum_{n\geq 1} \sigma_{k-1}(n)q^n.   
\]
We also use
\begin{align*}
   Q&\df 240G_4 = 1  + 240\sum_{n\geq1} \sigma_3(n)q^n, \\
   R&\df -504 G_6 =  1-504\sum_{n\geq 1} \sigma_5(n)q^n.
\end{align*}

The eta-function is
\[
  \eta(\tau)=\eta(q) = q^{1/24}\prod_{n=1}^{\infty} (1-q^n)  
\]
and the discriminant is
\[
    \Delta(\tau) = \eta(\tau)^{24} = \frac{1}{1728}(Q^3-R^2).
\]
Finally, the absolute modular invariant is
\[
    j(\tau) = \frac{Q^3}{\Delta}.
\]

\subsection{$p$-adic modular forms}
In \cite{Serre2} Serre defined the ring of $p$-adic modular forms of tame level one as the completion of the ring of modular forms $\QQ_p[Q,R]$ of level one with respect to the supremum norm taken on $p$-adic Fourier coefficients.\ This is a $p$-adic Banach algebra that we denote by $M_p$.\ Moreover
for any $k$ in weight space $M_{p, k}$ denotes the subspace of weight $k$ forms. The ring $M_p$ contains many new series, some of classical origin. For example, the $p$-adic Eisenstein series are defined for \emph{nonzero, even} $k\in X$ by
\[
 G_k^*(\tau) \df G_k(\tau) - p^{k-1}G_k(p\tau)      
\]
and we have $G_k^*\in M_{p, k}$. Notice that for even integers $k\geq 2$, each $G_k^*$ is a classical modular form on $\Gamma_0(p)$, whereas Serre showed that they are $p$-adic modular forms of tame level one. In fact, Serre showed more generally that every classical form on $\Gamma_0(p)$ with rational Fourier coefficients is a $p$-adic modular form of tame level one.

The Fourier expansion for this $p$-adic family of Eisenstein series can be reexpressed as 
\begin{equation}
\label{eq:Gk}
    G_k^* = \tfrac{1}{2}\zeta_p(1-k) + \sum_{n=1}^\infty \sigma_{k-1}^*(n) q^n,
\end{equation}
where $\zeta_p$ is the Kubota-Leopoldt $p$-adic zeta function \cite{Serre2,Washington} and
\[
\sigma_{k-1}^*(n)  = \sum_{\substack{d\mid n \\ \gcd(d,p) =1}} d^{k-1}.
\] 
The $p$-adic zeta function is an analytic function on the set of \emph{odd} elements of weight space (where it vanishes) and for \emph{even integers} $k\geq 2$ it satisfies
\begin{align*}
    \zeta_p(1-k) = -(1-p^{k-1})\frac{B_k}{k}.
\end{align*}
Note that $\zeta_p(s)$ has a pole of order $1$ at $s=1$, just as the complex zeta function does. Formula \eqref{eq:Gk}  shows that the Fourier coefficients of the series $G_k^*$ vary analytically over weight-space as a function of $k$. This idea was formalized by Wiles \cite{Wiles} as the standard example of a $\Lambda$-adic family of Eisenstein series; see also the book \cite{Hida} of Hida for an accessible introduction to this subject.

Katz \cite{Katz} described geometric foundations for the theory of $p$-adic modular forms that generalizes Serre's work. In this optic, $p$-adic modular forms are defined as sections of the vector-bundles underlying classical modular forms, but restricted to lie over certain rigid analytic subsets of modular curves deprived of discs around supersingular elliptic curves. Serre's theory arises, in essence, by removing discs of $p$-adic radius one around each supersingular elliptic curve (which are finite in number). By shrinking these discs, one focuses attention on smaller spaces of $p$-adic modular forms that may be better behaved from an arithmetic perspective, thanks to their improved radius of convergence. And indeeed, these spaces have a limit, called the space of \emph{overconvergent modular forms} (cf. section 3.5 of \cite{Vonk}). The overconvergent space possesses the useful property that the classical Hecke operator $U_p$ has a discrete spectrum when acting on the space of $p$-adic overconvergent modular forms, whereas this discreteness fails more generally for the full ring of $p$-adic modular forms of tame level one as defined by Serre.\ See \cite{Vonk} for a full and accessible discussion containing many concrete examples and references to the literature of the study of overconvergent modular forms and their applications in number theory.

\subsection{The $\Lambda$-ring}\label{ss:Lambda}
In this section we follow Chapter 7 of \cite{Washington} or, for slightly more generality, Chapter 3 of \cite{CoatesSujatha}. Let $\cG$ denote a profinite group, and let $\Lambda(\cG)$ denote the corresponding \emph{completed group algebra}, defined as the inverse limit
\[
\Lambda(\cG) \cong \varprojlim_{\cH \subseteq \cG} \ZZ_p[\cG/\cH],
\]
where the inverse limit is taken with respect to the set of open subgroups $\cH\subseteq \cG$, which are necessarily also closed and of finite index by compactness of $\cG$. There are natural transition maps and the corresponding inverse limit $\Lambda(\cG)$ is called the \emph{Iwasawa algebra} of $\cG$. This ring was first used by Iwasawa \cite{Iwasawa} in investigations of $p$-adic $L$-functions. In this application one takes $\cG = \ZZ_p$ or more generally $\cG = \ZZ_p^{\times}$, and constructs the $p$-adic zeta function by using compatibilities between Stickelberger elements in group rings 
\[
\ZZ_p[\ZZ_p/p^n\ZZ_p] \cong \ZZ_p[\ZZ/p^n\ZZ],
\]
cf. Chapter 6 of \cite{Washington}, to construct an element of $\Lambda(\ZZ_p)$ that \emph{is} the $p$-adic zeta function (or at least, one of its branches on weight space). For this reason the ring $\Lambda = \Lambda(\ZZ_p)$ is sometimes simply called the \emph{Iwasawa algebra}.\ Key to its r\^{o}le in defining $p$-adic $L$-functions as analytic functions is the fact that there is an isomorphism
\[
\Lambda \cong \pseries{\ZZ_p}{T}
\]
defined by the \emph{Mahler transform}, cf.\ Definition 3.3.2 and Theorem 3.3.3 of \cite{CoatesSujatha}.\ Thus, by the strong triangle inequality, elements of $\Lambda$ can be interpreted as analytic functions defined by power-series that converge for all $z \in p\ZZ_p$ (note that since the $p$-adic zeta function has a pole at $s=1$, it actually defines an element of the fraction field of $\Lambda$). 

In Section \ref{SLambda} below we define a family of states in the $p$-adic Heisenberg algebra with coefficients that are analytic functions on weight-space, and thus these coefficients can be interpreted as elements of $\Lambda(\ZZ_p^\times)$ by the Mahler transform, cf. Section 3.6 of \cite{CoatesSujatha}.\ The corresponding family maps onto the classical $\Lambda$-adic family of Eisenstein series discussed in \cite{Serre2}, and in this way we construct a VOA-theoretic avatar of this $\Lambda$-adic family of modular forms.

\subsection{Heisenberg algebras}\label{SS2.2}
Let $h(-m)$ denote independent indeterminates for all integers $m\geq 1$ and consider the infinite polynomial ring
\[
  S_{\alg}\df \QQ_p[h(-1),h(-2),\ldots].
\]
We endow $h(-m)$ with degree $m$, so that the degree $n$ graded piece of $S_{\alg}$ has dimension equal to the number of partitions of $n$.\ This ring $S_{\alg}$ can be equipped with the structure of a vertex operator algebra (over $\QQ_p$) called the \emph{rank-one Heisenberg algebra}, or \emph{Heisenberg VOA}.\ The subscript 'alg' is used to distinguish the classical VOA from its $p$-adic counterpart. 

We shall also frequently use a slightly different way to represent states in $S_{alg}$ in keeping with its origins as a highest weight module over the Heisenberg Lie algebra, and which is ubiquitous throughout VOA theory.\ Namely, $h$ is promoted from a mere cypher to a state of weight $1$ in $S_{alg}$, and its vertex operator is
\[
 Y(h, z) \df \sum_{n\in\ZZ} h(n)z^{-n-1},
\]
so that the modes $h(n)$ are operators on $S_{alg}$.\ They satisfy the canonical commutator relations of quantum mechanics, i.e., $[h(m), h(n)]= m\delta_{m+n, 0} Id$.\ The canonical vacuum state is
$\mathbf{1}$ and  $S_{alg}$ has a natural basis consisting of states
\begin{align}\label{statedef}
& h(-n_1)h(-n_2)...h(-n_t)\mathbf{1}& (n_1\geq n_2\geq...\geq n_t\geq 1).
\end{align}
The $p$-adic Heisenberg algebra (or VOA) $S$ is defined as a certain completion of $S_{alg}$.\ To describe this completion, for each real number $R\geq 1$ we introduce a norm on $S_{\alg}$ following \cite{FMpadic} as
\[
  \abs{\sum_{I} a_I h^I}_{R} \df \sup_I \abs{a_I}R^{\abs{I}},
\]
where $I$ runs over all finite multi-subsets of $\ZZ_{<0}$, and $\abs{I} = -\sum_{i \in I}i$. Let $S_R$ denote the corresponding completion of $S_{\alg}$, and write $S = S_1$. Then \cite[Proposition 9.1]{FMpadic} shows that
\[
  S_R = \left\{\sum_I a_Ih^I \in \pseries{\QQ_p}{h(-1),h(-2),\ldots} \mid \lim_{\abs{I} \to \infty} \abs{a_I}R^{\abs{I}} = 0\right\}.
\]
The space $S$ is the $p$-adic Heisenberg algebra, and if $R_1 < R_2$ then $S_{R_2} \subseteq S_{R_1}$. In particular, for all $R_1\geq R_2\geq 1$ we have \cite[Corollary 9.2]{FMpadic} containments
\[S_{\alg} \subseteq S_{R_1} \subseteq S_{R_2}\subseteq S.\]
The spaces $S_R$ will be significant for discussing certain aspects of the $p$-adic extension of the square-bracket formalism of $S_{\alg}$ discussed in Section \ref{Ssqbrack}.\  In particular, we will require the following simple lemma. Actually, we will at times require slight strengthenings of this Lemma, but we thought it useful to include a basic example first.
\begin{lem}
  \label{l:Sr1}
Let $v=\sum_{I}a_Ih^I \in S$, and suppose that for $\abs{I} \gg 0$ we have $a_I/(\abs{I}!) \in \ZZ_p$. Then $v\in S_R$ for all $R$ in the range $1 \leq R < p^{1/(p-1)}$.
\end{lem}
\begin{proof}
  By Legendre's formula, for integers $m \geq 0$ we have
  \[
  \nu(m!) \geq \frac{m}{p-1}-(p-1)\log_p(m).
\]
Therefore, for all multisets $I$ with $\abs{I}$ large enough, our hypothesis on $v$ states that
\[
  \nu(a_I) \geq \frac{\abs{I}}{p-1} - (p-1)\log_p(\abs{I}).
\]
Multiplying by $-1$ and raising to the $p$th power yields
\[
\abs{a_I} \leq (p^{1/(p-1)})^{-\abs{I}}\cdot p^{(p-1)\log_p(\abs{I})}
\]
Hence if $1 \leq R < p^{1/(p-1)}$ and we write $R = p^{\alpha}$ where $0 \leq \alpha < 1/(p-1)$, then we find that
\[
  \abs{a_I}R^{\abs{I}} \leq  p^{(\alpha - \frac{1}{p-1})\abs{I} + (p-1)\log_p(\abs{I})}.
\]
But, since linear growth outpaces logarithmic growth, and since $\alpha - \frac{1}{p-1} < 0$, this goes to zero as $\abs{I}$ grows, which proves the lemma.
\end{proof}

\section{The square bracket formalism and the operator $L[0]$}\label{Ssqbrack}
This Section, which to some extent is a continuation of the previous Section, deals with the so-called square bracket formalism, sometimes also refered to as VOAs on a cylinder, or
genus one VOAs.\ In particular, we look closely at a certain operator $L[0]$. The nature of its point spectrum is the fulcrum upon which the calculations of the present paper rest.

\subsection{The square bracket VOA}\label{SSsqVOA}
Although we are mainly concerned with the case of the rank $1$ Heisenberg VOA $S_{\alg}$, the general case is not much different so we work generally, at least at the outset.\ The basic idea was introduced by Zhu \cite{Zhu} and further discussion may be found in \cite{DLMModular, MT2}.\ It is important to point out that while these references all work over the field $\mathbf{C}$ the theory is unchanged if we work over any base field of characteristic $0$ such as $\QQ_p$, the main case of interest to us.

Given a VOA $(V, Y, \mathbf{1}, \omega)$ there is a second quadruple $(V, Y[\ ], \mathbf{1}, \tilde{\omega})$, where the underlying space $V$ as well as the vacuum vector $\mathbf{1}$ coincide and the square bracket conformal vector is
$\tilde{\omega}\df\omega-\frac{c}{24}\mathbf{1}$; here $c$ is the central charge of $V$, which is equal to $1$ for $S_{alg}$.\ The critical point here is the definition of the new vertex operator $Y[\ ]$, defined by
\begin{align}\label{defsqbk}
   Y[v, z] \df e^{kz}Y(v, e^z-1)=: \sum_{n\in \ZZ} v[n]z^{-n-1}\ \ (v \in V_k).
\end{align}
Here, we write $V=\oplus_k V_k$ and extend the definition of $Y[v, z]$ to all $v\in V$ by linearity in $v$.\ In particular, for the Virasoro vector the notation is
\[
  \sum_{n\in \ZZ} L[n]z^{-n-2}\df Y[\tilde{\omega}, z]. 
\]
For example, one can check from the definitions (loc.\ cit.) that
\begin{align}
 L[-1]&=L(0)+L(-1), \label{-1mode}\\
 L[0] &= \sum_{n\geq 0} \frac{(-1)^{n+1}}{n(n+1)}L(n).\label{0mode}
\end{align}
We shall make good use of both of these formulas.

Now the quadruple $(V, Y[\ ], \mathbf{1}, \tilde{\omega})$ is itself a VOA, indeed it is isomorphic to the original VOA $V$.\ This is not obvious, and was first proved in \cite{Zhu} in some special cases including the case at hand when $V=S_{\alg}$. The main utility of this fact for us right now is that we obtain a second integral grading on the space $V$, i.e., the square bracket conformal grading defined by
\[
 V = \oplus_k V_{[k]}   
\]
where $V_{[k]} \df \{v\in V \mid L[0]v=kv\}$.\ It is true \cite{DLMModular} that for each integer $n$ we have
\[
    \oplus_{k\leq n} V_k = \oplus_{k\leq n} V_{[k]}
\]
but in practice it is awkward to express a given state $v\in V_{[k]}$ as a linear combination of states in $V_n$ for $n\leq k$ and vice-versa.

\subsection{The character map}\label{SScharmap}
In this Subsection we will consider the trace functions and $q$-expansions associated to a VOA of central charge $c$.\ We begin with a general VOA that carries the conformal grading $V = \oplus_k V_k$.\ For a state $v\in V_k$, the \emph{zero mode} of $v$ is defined by 
\[
    o(v)= v(k-1)
\]
and extended by linearity to all $v\in V$.\ It is well-known that zero modes preserve the the homogeneous spaces $V_k$, that is $o(v): V_k \rightarrow V_k$. Then we can define a formal $q$-expansion as follows:
\[
 Z(v) = Z(v, \tau)=Z(v, q) \df \Tr_V o(v)q^{L(0)-c/24} = \sum_k \Tr_{V_k} o(v) q^{k-c/24}.
\]
The character map, or $1$-point function,  for $V$ is the  linear map
\[
  Z: V \rightarrow q^{-c/24}\QQ_p[q^{-1}][[q]]   
\]

For example, taking $v=\mathbf{1}$ and $V=S_{\alg}$ we have
\[
  Z(\mathbf{1}) = \frac{1}{\eta(\tau)}. 
\]
Continuing with this special case, the \emph{normalized character map} $f$ for $S_{alg}$ that appears in \eqref{Sdiag} is defined to be $\eta^{-1}Z$, so that
\[
   f(v)=\frac{Z(v)}{\eta}. 
\]

Now we come to the Mason-Tuite theorem \cite{MT2, MT} that gives a complete and explicit description of the character map for $S_{\alg}$.\ This comes in two parts.\ The first part
says that, as explained in the Introduction,  $f$ induces a linear map
\[
 f: \oplus_k (S_{alg})_{[k]}\longrightarrow \QQ_p[E_2, E_4, E_6].   
\]
Much more  is true, however.\ As discussed in the Introduction,  the normalized character map $f$ is graded in the sense that if $v\in V_{[k]}$ then 
$f(v) = g(\tau)$ for some quasimodular form $g(\tau)$ of weight $k$.\ This is a fundamental feature when  considering the corresponding $p$-adic VOAs.\ Finally, $f$ surjects onto the full algebra of quasimodular forms.

The second part of the Theorem is an \emph{explicit} formula for $Z(v)$.\ This is crucial for the applications in this paper.\
To describe this we need some notation.\ Let
$h\in (S_{\alg})_1$ be the canonical weight $1$ state (cf.\ Subsection \ref{SS2.2}).\ Then the corresponding square bracket VOA is also a rank $1$ Heisenberg VOA with the same canonical generator $h \in (S_{\alg})_{[1]}$. The square bracket analog of \eqref{statedef} is
\begin{align*}
  v&=h[-n_1]....h[-n_t].\mathbf{1},   &(n_1\geq n_2\geq ...\geq n_t\geq 1,\ k=\sum_i n_i)
\end{align*}
for a state $v\in(S_{\alg})_{[k]}$.\ This defines a basis of $S_{\alg}$ as we range over all partitions.\ Then we have
\begin{align}\label{Zform}
  f(v)   &= \sum_{\sigma}\prod_{(rs)}
 \widehat{E}_{n_r+n_s}(\tau)
\end{align}
where $\sigma\in \Sigma_t$ ranges over all fixed-point free involutions of the symmetric group $\Sigma_t$ of degree $t$, $(rs)$ ranges over the transpositions in $\Sigma_t$ whose product is equal to $\sigma$, and we are using the notational conventions of Subsection \ref{SSEis} for the Eisenstein series.

Some special cases of \eqref{Zform} will be particularly useful for us.\ The first already played a r\^{o}le in \cite[Theorem 10.1]{FMpadic}.
\begin{lem}\label{lemvr}
For an odd positive integer $r$ introduce the square bracket states 
\[
  v_r\df \frac{(r-1)!}{2} h[-r]h[-1]\mathbf{1}.        
\]
Then
\[ 
f(v_r) = G_{r+1}.
\]
\end{lem}

\begin{lem}\label{lemumnstate} For  positive integers $m$, $n$
introduce the square bracket states
\begin{align*}
 u_{m,n}\df& (-1)^{m+n}\frac{120^m1008^n}{(2m-1)!!(2n-1)!!} h[-2]^{2m}h[-3]^{2n}\mathbf{1}.
\end{align*}
Then
\[
  f(u_{m,n})=  Q^m R^n.
\]
\end{lem}
\begin{proof}
Use the formula \eqref{Zform} with $v=h[-2]^{2m}h[-3]^{2n}\mathbf{1}$.\ To be clear, this will involve a sum of terms each of which involves products of $\widehat{E}_{2+2}$, $\widehat{E}_{2+3}$, $\widehat{E}_{3+3}$.\ However $\widehat{E}_5=0$ so the formula produces a modular form that is a constant times $Q^mR^n$.\ The constant in question is determined by \eqref{wideE} and is equal to
\[
 (2m-1)!!(-6)^m (2n-1)!! (30)^n.
\]
All in all, this shows that  
\begin{align*}
 f(h[-2]^{2m}h[-3]^{2n}\mathbf{1}, \tau) =&(2m-1)!!(-6)^m (2n-1)!! (30)^n \left(\frac{1}{720}\right)^m \left( \frac{-1}{60.504}  \right)^nQ^mR^n\\
 =&(2m-1)!! (2n-1)!! \left(\frac{-1}{120}\right)^m \left( \frac{-1}{1008}  \right)^nQ^mR^n,
 \end{align*}
and the Lemma follows.
\end{proof}

\subsection{$X$-weights in the $p$-adic Heisenberg VOA}\label{SSXwtH}
We have already discussed two types of weights of states in $S_{\alg}$, namely the conformal weights of states given by eigenvalues of the round bracket operator $L(0)$ and similarly the square bracket weights
which are eigenvalues of $L[0]$.\ Such weights are always rational integers.\
Contemplation of diagram \eqref{ugdiag} gives rise to new notions of weights of states in the completion $S$.\ This was discussed in the Introduction and we shall take up this phenomenon in the next few  Subsections.

The \emph{$X$-weight} of a state in $S$ arises directly from Serre's result that $p$-adic modular forms carry a weight that lies in weight space $X$, cf.\ Subsection \ref{SSXwt}.\
Suppose that $(u_i)$ is a sequence of states in $S_{\alg}$  such that each $u_i$ has square bracket weight $k_i$, i.e., $L[0]u_i=k_iu_i$.\ And suppose further that $(u_i)$ is a Cauchy sequence in $S_{alg}$.\ Let $u\df\lim_ {i\rightarrow\infty} u_i\in S$.\ Because $f$ is $p$-adically continuous then
$(f(u_i))$ is a Cauchy sequence in $\QQ_p[E_2, E_4, E_6]$  where each term $f(u_i)$ has a weight $k_i$.\ The limit of $(f(u_i))$ is thus a Serre $p$-adic modular form.\ Then by Serre's work, the sequence $(k_i)$ is a Cauchy sequence and has a limit $k\in X$.\ This argument shows that the following definition is not vacuous.

\begin{dfn}\label{defXwt}
  Let $(u_i)$ be a Cauchy sequence of square bracket eigenstates as above and let $u\df\lim_{i\rightarrow\infty} u_i$.\ Then we call $k$ the \emph{$X$-weight} of the $p$-adic state $u$.
\end{dfn}
\subsection{Continuity of $L[0]$}
We continue to consider the $X$-weight of $u$ as in Definition
\ref{defXwt}.\ In stark contrast to the weights of a state in an algebraic VOA given by eigenvalues of $L(0)$ or $L[0]$, the $X$-weight of $u$ is not defined to be an eigenvalue of $L[0]$ (or any other operator for that matter).\ In order to relate $X$-weights to the point-spectrum of $L[0]$ we must resort to indirect means.\ The nub of the problem is this:\ we would like to  understand $L[0]v$, however this is not necessarily defined.\ This is due to the fact that, unlike its round bracket counterpart $L(0)$, the operator $L[0]$ is \emph{not a bounded operator on $S$}.\ Indeed it does not converge in the algebra of operators on $S_{alg}$ and in view of the expression \eqref{0mode} this is hardly surprising.\ This circumstance means that we cannot rely on $L[0]$ to map the Cauchy sequence $(u_i)$ to another convergent sequence.

A solution to this dilemma is to ascertain a large enough
closed subspace of $S$ on which $L[0]$ \emph{is} continuous, and then only work with Cauchy sequences $(u_i)$  in this subspace.\ At any rate this is the strategy we employ here.\ It motivates the next result, where we take the closed subspace to be one of the $p$-adic Banach spaces $S_r$ that we recalled in Subsection \ref{SS2.2}.
\begin{thm}\label{thmL}
$L[0]$ is a bounded operator on $S_r$ whenever $r\geq p^{1/p}$.
\end{thm}

Let us first emphasize that the norm $\abs{\cdot}_r$ on $S_r$ is not the same as that for $S=S_1$.\ It is stronger than $\abs{\cdot}_1$ in the sense that a Cauchy sequence in $S$ may not be Cauchy in $S_r$.\ We begin the proof of Theorem \ref{thmL} with
\begin{lem}\label{lemS} Suppose that $u\in S_{\alg}$.\ Then 
each mode $u(n)$ leaves $S_r$ invariant for all integers $n$
and all $r\geq 1$.
\end{lem}
\begin{proof} We employ the notation of Subsection \ref{SS2.2}.\ We may assume without loss that $u=h^J$ for some $J$.\ Then
$u(n)h^I=\sum_K m_{IK} h^K$ with each $m_{IK}\in\mathbf{Z}$ and $\abs{K}= \abs{I}+\abs{J}-n-1$.\
Let $v\in S_r$ with $v\df\sum_I a_Ih^I$.\ Then
$\lim_{|I|\rightarrow\infty} |a_I|r^{|I|}=0$.

Because $r\geq 1$ then $S_r\subseteq S_1$ and  $u(n)$ preserves limits in $S_1$.\ Therefore
\[
    u(n)v = \sum_I a_I u(n)h^I =\sum_I\sum_K
 a_I m_{IK} h^K.
\]
Then for fixed $n$ and $J$,
\begin{align*}
  \lim_{\abs{K}\rightarrow\infty} \abs{\sum_{I} a_Im_{IK}} r^{|K|} &= \lim_{\abs{I}\rightarrow \infty} \abs{\sum_{I} a_Im_{IK}} r^{\abs{I}+\abs{J}-n-1} \\
 &\leq r^{\abs{J}-n-1}\lim_{\abs{I}\rightarrow\infty}\sup_{\abs{I}}\abs{a_I}r^{\abs{I}} = 0.
\end{align*}
This shows that $u(n)v\in S_r$, thus proving the Lemma.
\end{proof}

\begin{rmk} The Proposition says that each $S_r\ (r\geq 1)$ is a \emph{weak module} for $S_{\alg}$.\ That is, a module with no assumed properties \emph{vis a vis} conformal grading.
\end{rmk}

\begin{proof}[Proof of Theorem \ref{thmL}] Let $w\df\sum_I a_Ih^I$ be a state in $S_r$ for some $r\geq 1$.\ Thus 
\[\lim_{\abs{I}\rightarrow\infty} \abs{a_I}r^{\abs{I}}=0.\] 
Each Virasoro mode $L(n)$ leaves $S_r$ for $r\geq1$ invariant by Lemma \ref{lemS}.\ Thus using equation \eqref{0mode}, in order to prove the Theorem it suffices to show that the operators 
$\tfrac{1}{n(n+1)}L(n)$ have uniformly bounded operator norms on $S_r$ for suitable $r$.

Now $\abs{w}_r\df \sup_{I} \abs{a_I}r^{\abs{I}}$. Set
$L(n)h^I=\sum_K m_{nIK} h^K$ with $m_{nIK}\in\ZZ$ and $\abs{K}= \abs{I}-n$. Then we must consider
\begin{align*}
\sup_{n\geq1} \abs{\tfrac{1}{n(n+1)}L(n)w)}_r  =& \sup_{n\geq1} \abs{\sum_{I, K} \frac{a_Im_{nIK}}{n(n+1)} h^K}_r\\ 
\leq & \sup_{n\geq1} \sup_{I}\abs{ \frac{a_I}{n(n+1)}}r^{\abs{I}-n} =\sup_n r^{-n} abs{\frac{1}{n(n+1)}}\abs{w}_r.
\end{align*}
Thus we have to show that for any fixed $r\geq p^{1/p}$, the expression 
\[
  E_n\df r^{-n}\abs{\frac{1}{n(n+1)}}  
\]
is uniformly bounded for $n\geq 1$.\ If
$n(n+1)$ is coprime to $p$ then $E_n=r^{-n}\leq p^{-n/p}\leq 1$.\ Suppose that $n+1=p^km$ where $k\geq1$ and $m$ is an integer coprime to $p$.\ Then
$E_n = r^{-n}p^k\leq r^{-n}p^{(n+1)/p}\leq p^{-n/p+(n+1)/p}=p^{1/p}$.\ Finally, if $p\mid n$ then by a
similar argument we again get an (even smaller) upper bound.\ We skip the details.\ Thus we have $\abs{E_n}\leq p^{1/p}$ for all $n\geq 1$, and this completes the proof of Theorem \ref{thmL}.
\end{proof}

As previously discussed, we now have
\begin{thm}\label{thmsuwt} Let $r$ be a real number satisfying $r\geq p^{1/p}$.\ Suppose that $(u_i)$ is a  sequence of states in $S_{\alg}$ satisfying $L[0]u_i=k_iu_i$ and assume further that
$(u_i)$ is a Cauchy sequence in $S_r$ with limit $u$.\
 Then $u\in S_r$ has an $X$-weight, say $(s, u)\in\ZZ_p\times\ZZ/(p-1)\ZZ$ (cf.\ Subsection \ref{SSXwt}) and we have
\[
    L[0]u = su.
\]
\end{thm}
\begin{proof} We first point out that because $(v_i)$ is Cauchy in $S_r$ then it is also Cauchy in $S_1$ and it follows from the discussion in Subsection \ref{SSXwt} that $v$ has an $X$-weight $k=(s, u)$ which is the limit of the sequence $(k_i)$ in $X$.\
Using Theorem \ref{thmL} we calculate
\[
  L[0]u = \lim_{i\rightarrow \infty} L[0]u_i =\lim_{i\rightarrow\infty} k_iu_i
  = \left(\lim_{i\rightarrow\infty} k_i\right)v = su
\]
where the last limit is in $\ZZ_p$.\ The Theorem is proved.
\end{proof}

\begin{rmk} We shall see in Section \ref{SLambda} that \emph{every} weight $k\in X$ occurs as the $X$-weight of a state in $S_r$ for choices of $r$ that permit application of Theorem \ref{thmsuwt}.\ Then we shall be able to deduce that every $p$-adic integer lies in the point spectrum of $L[0]$.
\end{rmk}

We complete this Section with a related Lemma that we will use later.
\begin{lem}\label{lemhnbnd} For each  integer $n$ and each real number $r\geq 1$,
$h(n)$ is a bounded operator on $S_r$ satisfying
\[
\abs{h(n)}_r \leq r^{-n}.    
\]
\end{lem}
\begin{proof}
Note that $h(n)$ acts on $S_r$ by Lemma \ref {lemS}.\ As before, let $w\df\sum_I a_Ih^I \in S_r$, so that $\abs{w}_r = \sup_{I} \abs{a_I}r^{\abs{I}}$ and suppose first that $n>0$.\ Then there are rational integers $e_I$ such that $h(n)w = n\sum_J a_I e_I h^J$ and $J$ ranges over those index sets satisfying $-n\in I$ and
 $J=I\setminus\{-n\}$.\
Then we find that
\[
  \abs{h(n)w}_r = \abs{n\sum_J a_I e_I h^J}_r  \leq sup_I\abs{a_I}r^{\abs{I}-n}
  =\abs{w}_r r^{-n}.
\]
This completes the proof of the Lemma in  case $n>0$.\ If $n<0$ the proof is very similar but easier because $h(n)$ is then a creation operator.\ If $n=0$ then $h(0)=0$ and the result is obvious.\ This completes the proof of the Lemma in all cases.
\end{proof}

\section{A $\Lambda$-adic family of states}\label{SLambda}
We begin with the square bracket states $v_r\in S_{\alg}$ defined in Lemma \ref{lemvr}.\ Let us recall from Section 10, and especially Corollary 10.3 of \cite{FMpadic} that we have 
\begin{align*}
 u_r  &\df \frac{(1-p^r)}{2}\sum_{m=0}^{\infty}c(r,m)h(-m-1)h(-1)\mathbf{1} - \frac{(1-p^r)}{2}\frac{B_{r+1}}{r+1}\mathbf{1} = (1-p^r)v_r
\end{align*}
where 
\[
  c(r,m) \df \sum_{j=0}^m (-1)^{m+j}\binom mj (j+1)^{r-1} = m! \stirling{r}{m+1}.
\]
(Cf.\ Subsection \ref{SSspecnos} for notation.)\ Notice that if $r$ converges to some $p$-adic value through a sequence of positive integers, then $u_r$ will have a $p$-adic limit if, and only if, $v_r$ does, in which case the two limits are equal.

 In order to effect an analytic continuation of $c(r, m)$ to all of weight space, we make two adjustments above. First define
\[
  c_p(r,m) \df \sum_{\substack{j=0\\ p \nmid j+1}}^m (-1)^{m+j}\binom mj (j+1)^{r-1}.
\]
\begin{lem}
  \label{l:cpvaluation}
For all $r \in X$ and $m \in \ZZ_{\geq 0}$ we have $c_p(r,m) \in m!\ZZ_p$.
\end{lem}
\begin{proof}
Notice that if $r$ is a positive integer then $c(r,m) \in m!\ZZ$. Let $r \in X$, and let $r_n$ denote an increasing sequence of positive integers that converges to $r$ in weight space.\ Then in particular $r_n \equiv r \pmod{p-1}$ for all $n$.\ Then the terms $(j+1)^{r_n-1}$ for $p \mid (j+1)$ converge to $0$ in $\ZZ_p$ as $n$ grows, and we see that \[\lim_{n \to \infty} c(r_n,m) = c_p(r,m),\] by Euler's theorem that $(j+1)^{\phi(p^{n})} \equiv 1 \pmod{p^n}$ when $\gcd(p,j+1) = 1$.\ Since each term $c(r_n,m)$ is contained in the closed subset $m!\ZZ_p$ of $\ZZ_p$, so is the limit $c_p(r,m)$.\ This concludes the proof.
\end{proof}

Next, for each $r \in X\setminus\{-1\}$ we write
\begin{equation}\label{grdef}
  g\langle r+1\rangle \df \frac 12\sum_{m=0}^\infty c_p(r,m)h(-m-1)h(-1)\mathbf{1}+\frac{1}{2}\zeta_p(-r)\mathbf{1},
\end{equation}
The coefficients of $g\langle r+1\rangle$ can be interpreted as analytic functions in the $\Lambda$-ring via the Mahler transform, as discussed in Subsection \ref{ss:Lambda}. Thus we view $g\langle r+1\rangle$ as a $p$-adic analytic family of states in the Heisenberg algebra. 

We will see below that $g\langle r+1\rangle$ arises as the limit of $v_{r}$ if $r+1$ converges to a point in weight space through a sequence of positive integers of increasing size.\ Notice that each term $c_p(r,m)$ is continuous for all $r \in X$, while the vacuum term is continuous on $X\setminus\{-1\}$.\ Initially this is nothing but a formal series, but observe the following:
\begin{prop}\label{prop4}
  For all real numbers $r$ in the range $1\leq r < p^{1/(p-1)}$ we have
  \[g\langle r+1\rangle \in S_r.\]
\end{prop}
\begin{proof}
  The monomial $h(-m-1)h(-1)$ is equal to $h^I$ for $I = (-m-1,-1)$ with $\abs{I} = m+2$. Write $R= p^\alpha$ for $\alpha < \tfrac{1}{p-1}$ as in the proof of Lemma \ref{l:Sr1}, and notice that in this case Lemma \ref{l:cpvaluation} and Legendre's theorem for $\nu_p(m!)$ gives us:
\begin{align*}
  \abs{c_p(r-1,m)}R^{\abs{I}} \leq p^{-\frac{m}{p-1}+(p-1)\log_p(m)}\cdot p^{\alpha(m+2)}\\
  \leq p^{(\alpha - \frac{1}{p-1})m + (p-1)\log_p(m) + 2\alpha}
\end{align*}
In particular, since $\alpha - \frac{1}{p-1}< 0$, the linear term in $m$ of the exponent dominates as $m$ grows and $\lim_{m \to \infty} c_p(r-1,m)R^{\abs{I}}=0$, as we wanted to show. 
\end{proof}

The next result explains why we have chosen to use the notation $g\langle r+1\rangle$ for this interpolated family of $p$-adic states.
\begin{thm}
\label{t:mainLambda}
  Let $\hat{f}$ denote the renormalized character map for the $p$-adic Heisenberg algebra $S$ (cf.\ (\ref{Sdiag})).\ Then for all odd weights $r \in X \setminus\{-1\}$, we have
  \[\hat{f}(g\langle r+1\rangle) = G_{r+1}^*\]
  and $\Res_{r=-1}g\langle r+1\rangle = \tfrac{1}{p}-1$.
\end{thm}
\begin{proof}
  For the residue computation, notice that each $c_p(r,m)$ term is analytic on all of $X$, and so the residue comes entirely from $\zeta_p(-r)$, which is known to have a residue of $\tfrac 1p-1$ at $r=-1$.

  For the character computation we proceed by $p$-adic approximation following \cite{Serre2}, Example 1.6.\ Choose a sequence $r_n\geq 3$ of increasing positive integers that converge to $r$ in weight space, such that $r_n \equiv r\pmod{\phi(p^{n+1})}$ and also $r_n \geq n+1$ for all $n$.\ Then as in \cite{Serre2}, top of page 206, $G_{r_n+1} \to G^*_{r+1}$.\ We have $f(v_{r_n}) = G_{r_n+1}$ by Lemma \ref{lemvr} so our result will follow by continuity of the map $f$ if we can show that the sequence $v_{r_n}$ converges to $g\langle r+1\rangle$.

 The vacuum term is handled by theory of the $p$-adic zeta function as described in Example 1.6 of \cite{Serre2}. Thus it remains to show that the terms $(1-p^{r_n})c(r_n,m)$ converge to $c_p(r,m)$ uniformly in $m$. Euler's identity $x^{\phi(p^{m+1})} \equiv 1 \pmod{p^{m+1}}$ whenever $p \nmid x$ implies that
  \[
  (1-p^{r_n})c(r_n,m) \equiv c_p(r,m) \pmod{p^{n+1}}
\]
for all $m$. Thus we obtain the desired convergence $v_{r_n} \to g\langle r+1\rangle$ and this concludes the proof.
\end{proof}

Now we can prove
\begin{thm}\label{thmL0eigen}
Let $r$ be a real number satisfying 
$p^{1/p}\leq r<p^{1/(p-1)}$.\ Then the following hold:
\begin{enumerate}
\item[(a)] For any weight $t\in X$ there is a nonzero state $u\in S_r$ that has $X$-weight $t$.
\item[(b)] For any $s\in \ZZ_p$ there is a nonzero state $u\in S_r$ such that  $L[0]u=su$.
\end{enumerate}
\end{thm}
\begin{proof}
Assuming the truth of part (a), (b) is an immediate consequence of Theorem \ref{thmsuwt}. It is here that we use the condition $r\geq p^{1/p}$.

As for (a), we first treat the case of \emph{even} weights $t$.\ Indeed, if 
$t\neq 0$ then we may take $u=g\langle t\rangle$ by Theorem \ref{t:mainLambda}.\ On the other hand, if $t=0$ we can simple choose $u=\mathbf{1}$.\ It remains to prove (a) in the case that $t$ is an \emph{odd} weight.\ With this in mind, assume that $(u_i)$ is a Cauchy sequence in $S_r$ with $u\df\lim_{i\rightarrow i} u_i$ such that $u$ has an \emph{even} $X$-weight, call it $k\in X$.\ We have already seen that such limit states exist in $S_r$.\ We will show that as long as $u\neq\mathbf{1}$ then $L[-1]u$ is a nonzero state in $S_r$ with corresponding $X$-weight $1+k$.\ This will suffice to complete the proof of the Theorem for all weights except $t=1$.\ But here we can simply choose the weight $1$ generator $h=h[-1]\mathbf{1}$ of $S_{\alg}$.

Our jumping-off point is the  expression \eqref{-1mode} for
the operator $L[-1]$.\ Because both $L(-1)$ and $L(0)$ are bounded operators on $S$ then the same is true of $L[-1]$.\ In particular, if $(u_i)$ is as above then $L[-1]u=\lim_{i\rightarrow\infty} L[-1]u_i\in S_r$.\ To verify that the $X$-weight of this limit state is $1+k$, we proceed as follows.\ Because the square bracket vertex operators close on an (algebraic Heisenberg) VOA, then in particular the square bracket Virasoro modes $\{L[n]\}$ close on a Virasoro algebra of central charge $1$.\ What we really need from this is the identity $[L[0], L[-1]] = L[-1]$.\ Thus if
$u_i$ has $L[0]$-weight $k_i$ then
\begin{align*}
 L[0]L[-1]u_i&= [L[0], L[-1]]u_i +L[-1]L[0]u_i \\
 &= L[-1]u_i+k_iL[-1]u_i\\
 &=(1+k_i)L[-1]u_i.
\end{align*}
So each $L[-1]u_i$ has $L[0]$-conformal weight $1+k_i$, and this immediately implies that
$L[-1]u$ has $X$-weight $1+k$.

It remains to show that $L[-1]u\neq0$.\ This follows from Lemma \ref{lemkerL[-1]} below, thereby completing the proof of the Theorem.
\end{proof}

\begin{rmk} 
We point out that the states $L[-1]u$ above of odd $X$-weight always satisfy $\hat{f}(L[-1]u)=0$ so that we do not get additional $p$-adic modular forms by this route. The reason for this is that it is well-known that for $u_i\in S_{\alg}$ we always have $f(L[-1]u_i)=0$. Hence this vanishing of characters remains true after taking limits, by continuity of the character map.
\end{rmk}

\begin{lem}\label{lemkerL[-1]} If $u\in S$ satisfies $L[-1]u=0$
then $u$ is a scalar multiple of $\mathbf{1}$.
\end{lem}
\begin{proof}
As explained in the proof of \cite[Proposition 7.3]{FMpadic}, we may write
\[ u = \sum_{i\geq 0} u_i   
\]
where each $u_i$ is a state in $(S_{alg})_i$
and $u_i\rightarrow 0$. We calculate
\[
   0=  L[-1]u = \sum_i L[-1]u_i = \sum_i (L(-1)+L(0))u_i
  =\sum_i (L(-1)u_i+iu_i).
\]

If the Lemma is false then there is a \emph{least positive} integer $j$ such that $u_j\neq0$.\ We will derive a contradiction.\ Since $L(-1): (S_{\alg})_i\rightarrow (S_{\alg})_{i+1}$ it follows from the previous display that $ju_j$ is equal to a linear combination of states $u_i$ of weight \emph{greater} that $j$.\ Since the $u_i$ are linearly independent and $j\neq0$, this can only happen if $u_j=0$, and this is the desired contradiction.
\end{proof}


\section{The character map in weight zero}
\label{s:weightzero}

Recall from the arithmetic theory of elliptic curves that there are finitely many supersingular $j$-invariants for each prime $p$, and they are all contained in $\FF_{p^2}$. When $p=2,3,5$ we have that $j = 0$ is the only supersingular $j$-invariant and it will thus suffice below to work over $\QQ_p$ for such primes. More generally, since some supersingular $j$-invariants may be defined over $\FF_{p^2}$ rather than over $\FF_p$ for arbitrary $p$, in general one needs to work over the quadratic unramified extension of $\QQ_p$. The theory of \cite{FMpadic} extends to this setting without change. 

The following example is discussed at the bottom of page 202 of \cite{Serre}. 
\begin{prop}
  \label{p:wt0}
  When $p=2,3,5$, we have
  \[M_{p,0} = \QQ_p\langle j^{-1}\rangle.\]
\end{prop}

In order to utilize Proposition \ref{p:wt0}, we now restrict to $p=2,3,5$, and eventually we will in fact simply take $p=2$ for simplicity.

First we recall how to view $j^{-n}$ as a $p$-adic modular form of level $1$ and weight $0$.\ Bearing in mind our notation for Eisenstein series (cf.\  Subsection \ref{SSEis}), for $p=2,3,5$ we visibly have $Q\equiv 1 \pmod{p}$. Thus,
\[
Q^{-1} = \lim_{m\to \infty} Q^{p^m-1}.
\]
It follows that
\[
 j^{-n} = \Delta^n Q^{-3n} = \lim_{m\to \infty}\Delta^nQ^{3n(p^m-1)}.
\]
Suppose that we can find states
\[
  J_{n,m} \in S_{\alg}
\]
such that the following properties hold:
\begin{enumerate}
    \item $f(J_{n,m}) = \Delta^n Q^{3n(p^m-1)}$;
    \item $J_n \df \lim_{m \to \infty} J_{n,m}$ exists for each $n$;
    \item there exists a bound $B$ such that $\abs{J_{n}} \leq B$ for all $n$.
\end{enumerate}
Assuming that these properties hold,  by continuity of the $p$-adic character map we will have
\[
\hat{f}(J_n) = \hat{f}(\lim_{m} J_{n,m}) = \lim_m f(J_{n,m}) = \lim_m\left(\Delta^n Q^{3n(p^m-1)}\right) = j^{-n}.
\]
It will then follow that if $\sum_{n\geq 0} b_nj^{-n} \in M_{p,0}$ for a sequence of scalars $b_n$ converging to $0$, then the state $\sum_{n} b_n J_n$ exists in $S$ and $\hat{f}(\sum b_n J_n) = \sum b_nj^{-n}$. Therefore, if we can find states $J_{n,m} \in S_{\alg}$ with properties (1)--(3) above, then we will have established surjectivity of the map $\hat{f}$ in weight zero. 

As it is, we will not quite achieve this goal. Instead we will  establish an estimate slightly weaker than (3). This will suffice to obtain many new $2$-adic modular forms in the image of $\hat{f}$, though not all. A precise statement is given below as Corollary \ref{c:main2adic} and reiterated in Theorem \ref{t:main1}(2) of the introduction.

\subsection{Some special states in $S_{\alg}$}
In the following $u_{m, n}$ refers to the states introduced in Lemma \ref{lemumnstate}.
\
\begin{lem}
  For nonnegative integers $a$, $b$ introduce the square bracket state
\begin{align*}
 U_{ab} \df& (-1)^b (588)^a(120)^b \  \times\\
&\sum_{i=0}^a
 \binom{a}{i}\frac{1}{(6i+2b-1)!!(4a-4i-1)!!}
 \left(\frac{250}{147}  \right)^i
 h[-2]^{6i+2b}h[-3]^{4a-4i}\mathbf{1}.
\end{align*}
Then
\[
f(U_{ab}) = \Delta^a Q^b.
\]
\end{lem}
\begin{proof}
Using Lemma \ref{lemumnstate} we obtain
\begin{align*}
& \Delta^a Q^b = \left( \frac{1}{12} \right)^{3a}
(Q^3-R^2)^a Q^b \\
=& 12^{-3a}\sum_{i=0}^a
(-1)^i \binom{a}{i} Q^{3i+b}R^{2a-2i}\\
=& 12^{-3a}\sum_{i=0}^a
(-1)^i \binom{a}{i}
Z(u_{(3i+b),(2a-2i})\\
=&12^{-3a}\sum_{i=0}^a
\binom{a}{i}\left\{(-1)^{b}\frac{(120)^{3i+b}(1008)^{2(a-i)}}{(6i+2b-1)!! (4(a-i)-1)!!} Z(h[-2]^{6i+2b}h[-3]^{4(a-i)}\mathbf{1})\right\}\\
=&(-1)^b (588)^a(120)^b \eta\ \  \times\\
&\sum_{i=0}^a
 \binom{a}{i}\frac{1}{(6i+2b-1)!!(4a-4i-1)!!}
 \left(\frac{250}{147}  \right)^i
 Z(h[-2]^{6i+2b}h[-3]^{4a-4i}\mathbf{1}).
\end{align*}
\end{proof}

This computation suggests that we set
\[
J_{n,m} \df u_{n,3n(p^m-1)}
\]
so that $a=n$ and $b=3n(p^m-1)$. Thus,
\begin{align}
  \label{eq:Jnm}
&J_{n,m} =(-1)^{n(p+1)}(588 \cdot 120^{3(p^m-1)})^n \nonumber \\ &\sum_{i=0}^n \frac{\binom{n}{i} 250^i}{(147)^i \cdot (6(i + np^m-n)-1)!! \cdot (4(n-i)-1)!!} h[-2]^{6(i+n(p^m-1))}h[-3]^{4(n-i)}\bone.
\end{align}
As discussed previously, if we can show that the limit
\[
J_n \df \lim_{m \to \infty} J_{n,m}
\]
exists, and that the series $J_n$ are bounded with $n$, then we will obtain new $p$-adic modular forms of weight zero in the image of the character map. The most complicated term in the definition of $J_{n,m}$ that involves the limit variable $m$ is the power $h[-2]^{6i+6n(p^m-1)}$. Therefore, in the next section we give a detailed study of these square-bracket states and their $p$-adic properties.

\section{Powers of $h[-2]$}\label{Sp=2I}
\label{s:hminus2}
Recall that the modes $h[n]$ are  defined in \eqref{defsqbk} via the formal series
\[
  Y[h,z] \df \sum_{n\in\ZZ} h[n]z^{-n-1} = e^zY(h,e^z-1)
\]
where we have taken $k=1$ because $h\in (S_{\alg})_1$.\ Of course, by definition of the algebraic Heisenberg algebra, $Y(h,z) = \sum_{n\in \ZZ} h(n)z^{-n-1}$.\ Therefore,
\begin{align*}
  h[-2] &= \Res_z(z^{-2} e^zY(h,e^z-1))\\
         &=\Res_z\left(z^{-2} e^z\sum_{n\in\ZZ}h(n)(e^z-1)^{-n-1}\right)\\
         &=\Res_z\left(z^{-2} e^z\sum_{n\geq -2}h(n)(e^z-1)^{-n-1}\right)\\
  &=\sum_{n\geq -2}h(n)\Res_z\left(\frac{e^z}{z^2(e^z-1)^{n+1}}\right)
\end{align*}
Recalling the definition of generalized Bernoulli polynomials from Subsection \ref{SSspecnos}, we have
\[
\frac{e^z}{z^2(e^z-1)^{n+1}}=z^{-n-3}e^z\left(\frac{z}{e^z-1}\right)^{n+1}=\sum_{m\geq 0} B_{m}^{(n+1)}(1)\frac{z^{m-n-3}}{m!}
\]
The residue arises when $m=n+2$. This shows that
\begin{equation}
  \label{eq:hminus2}
  h[-2] = \sum_{n\geq -2}\frac{B^{(n+1)}_{n+2}(1)}{(n+2)!}h(n).
\end{equation}
More generally for $r \geq 2$:
\begin{equation}
  \label{eq:hminusr}
  h[-r] = \sum_{n\geq -r}\frac{B^{(n+1)}_{n+r}(1)}{(n+r)!}h(n).
\end{equation}

\begin{rmk}
In principle these generalized Bernoulli numbers can have denominators that are divisible by $p$. Computations suggest that the Clausen-von-Staudt theorem generalizes as follows:
\[
  p^{\floor{\log_2(m+1)}} B_n^{(m)}(1) \in \ZZ_p.
\]
We have proved the following slightly weaker form of this:
\[
  p^{\floor{\log_p(mn+1)}}B_n^{(m)}(1) \in \ZZ_p.
\]
We will not need such estimates below.
\end{rmk}

Now, we want to take powers of these square-bracket states. We notice that $h(n)$ for $n\geq 3$ all commute with each other, and they commute with $h(-2)$, $h(-1)$, $h(0)$, $h(1)$ and $h(2)$. Therefore, let us write:
\begin{align*}
  A &=\sum_{n= -2}^3\frac{B^{(n+1)}_{n+2}(1)}{(n+2)!}h(n),\\
  B &= \sum_{n\geq 4}\frac{B^{(n+1)}_{n+2}(1)}{(n+2)!}h(n).
\end{align*}
Then $A$ and $B$ commute, so that we have
\[
  h[-2]^u = (A+B)^u = \sum_{r=0}^u \binom{u}{r}A^rB^{u-r}
\]
Eventually we need to consider $h[-2]^uh[-3]^v\bone$. This will be an element in the algebraic Heisenberg with a messy description. To evaluate this, let us also write
\begin{align*}
  C \df\sum_{n= -3}^3\frac{B^{(n+1)}_{n+3}(1)}{(n+3)!}h(n),\  \
  D \df \sum_{n\geq 4}\frac{B^{(n+1)}_{n+3}(1)}{(n+3)!}h(n).
\end{align*}
Then we have $[C,B]=0$, $[C,D]=0$, $[D,B]=0$, $[D,A]=0$ and hence
\begin{align*}
  h[-2]^uh[-3]^v &= \sum_{r=0}^u\sum_{t=0}^v\binom{u}{r}\binom{v}{t} A^rB^{u-r}D^{v-t}C^t
  &=\sum_{r=0}^u\sum_{t=0}^v\binom{u}{r}\binom{v}{t} A^rC^tB^{u-r}D^{v-t}
\end{align*}
Since the operators $B$ and $D$ annihilate the vacuum vector $\bone$, it follows that
\begin{equation}
\label{eq:mon1}
  h[-2]^uh[-3]^v\bone = A^uC^v\bone.
\end{equation}

In order to evaluate this, let us observe that:
\begin{align*}
  A &= h(-2) + h(-1) + \tfrac{1}{12}h(0)-\tfrac{1}{240}h(2)+\tfrac{1}{240}h(3),\\
  C&= h(-3)+\tfrac 32 h(-2)+ \tfrac 12 h(-1)+\tfrac{1}{240}h(1)-\tfrac{1}{480}h(2)+\tfrac{1}{945}h(3)
\end{align*}
As operators for $r \geq 1$ we have $h(r) = r\tfrac{\partial}{\partial h(-r)}$ and $h(0) = 0$ so that this simplifies to
\begin{align*}
  A &= h(-2) + h(-1)-\tfrac{1}{120}\tfrac{\partial}{\partial h(-2)}+\tfrac{1}{80}\tfrac{\partial}{\partial h(-3)},\\
  C&= h(-3)+\tfrac 32 h(-2)+ \tfrac 12 h(-1)+\tfrac{1}{240}\tfrac{\partial}{\partial h(-1)}-\tfrac{1}{240}\tfrac{\partial}{\partial h(-2)}+\tfrac{1}{315}\tfrac{\partial}{\partial(-3)}.
\end{align*}

If we combine equations \eqref{eq:Jnm} and \eqref{eq:mon1} we obtain the expression:
\begin{align}
  \label{eq:Jnm2}
J_{n,m} =&(-1)^{n(p+1)}(588 \cdot 120^{3(p^m-1)})^n \nonumber\\ &\sum_{i=0}^n\binom{n}{i} \frac{250^i}{(147)^i \cdot (6n(p^m-1)+6i-1)!! \cdot (4n-4i-1)!!} A^{6i+6n(p^m-1)}C^{4n-4i}\bone.
\end{align}

To begin our analysis, notice the following:
\begin{lem}
  The partial differential operators $A$ and $C$ commute.
\end{lem}
\begin{proof}
This can be checked directly, say on a computer, or it can be deduced from the fact, discussed in Subsection \ref{SSsqVOA}, that the algebraic VOA structures on the Heisenberg algebra using either the square or round bracket modes are isomorphic.\ Since $h(-2)$ and $h(-3)$ commute by definition, this  result  implies that $h[-2]$ and $h[-3]$ also commute, and the Lemma is a consequence of this.
\end{proof}

The preceding Lemma implies that while studying the limit over $m$ of equation \eqref{eq:Jnm2}, we may apply powers of $A$ before applying powers of $C$.\ To this end, let us define a recursive sequence of polynomials $p_n$ by setting $p_0 = \bone$ and $p_n = Ap_{n-1}$ for $n\geq 1$.\ For simplicity let us now write $a = h(-1)$ and $b = h(-2)$.\ Then with this notation, $A$ acts as the operator $a+b-\tfrac{1}{120}\tfrac{d}{da}$.
\begin{prop}
  \label{p:pk}
  We have for $k\geq 0$:
  \begin{align*}
    p_{2k} =&\sum_{i=0}^k(2(k-i)-1)!!(-120)^{i-k}\binom{2k}{2(k-i)}(a+b)^{2i},\\
    p_{2k+1} =&\sum_{i=0}^k(2(k-i)-1)!!(-120)^{i-k}\binom{2k+1}{2(k-i)}(a+b)^{2i+1}.
  \end{align*}
\end{prop}
\begin{proof}
  The proof is by induction. First notice that
  \begin{align*}
    Ap_{2k} =& \sum_{i=0}^k(2(k-i)-1)!!(-120)^{i-k}\binom{2k}{2(k-i)}(a+b)^{2i+1} + \\
    &\sum_{i=1}^k(2(k-i)-1)!!(-120)^{i-k-1}(2i)\binom{2k}{2(k-i)}(a+b)^{2i-1}\\
=& \sum_{i=0}^k(2(k-i)-1)!!(-120)^{i-k}\binom{2k}{2(k-i)}(a+b)^{2i+1} + \\
             &\sum_{i=0}^{k-1}(2(k-i)-3)!!(-120)^{i-k}(2i+2)\binom{2k}{2(k-i-1)}(a+b)^{2i+1}\\
    =&(a+b)^{2k+1}+ \sum_{i=0}^k(2(k-i)-3)!!(-120)^{i-k}\cdot\\
    &\left\{(2(k-i)-1)\binom{2k}{2(k-i)}+(2i+2)\binom{2k}{2(k-i-1)}\right\}(a+b)^{2i+1}
  \end{align*}
  Since
  \begin{align*}
    &(2(k-i)-1)\binom{2k}{2(k-i)}+(2i+2)\binom{2k}{2(k-i-1)}\\
    =&(2(k-i)-1)\frac{(2k)!}{(2(k-i))!(2i)!}+(2i+2)\frac{(2k)!}{(2(k-i-1))!(2(i+1))!}\\
    =&(2(k-i)-1)\left\{\frac{(2k)!}{(2(k-i))!(2i)!}+\frac{(2k)!}{(2k-2i-1)!(2i+1)!}\right\}\\
    =& (2(k-i)-1)\left\{\binom{2k}{2i}+\binom{2k}{2i+1}\right\}\\
    =&(2(k-i)-1)\binom{2k+1}{2(k-i)}
  \end{align*}
  Hence we indeed have $Ap_{2k} = p_{2k+1}$ by induction.\ The proof for $Ap_{2k+1}= p_{2k}$ is analogous.
\end{proof}

\begin{cor}
  \label{c:Apowers}
  For all $n,m \geq 1$ we have
  \begin{align*}
    & (-1)^{n(p+1)}120^{3n(p^m-1)}A^{6n(p^m-1)}\bone\\
   =& \sum_{j=0}^{3n(p^m-1)}(6n(p^m-1)-2j-1)!!(-120)^{j}\binom{6n(p^m-1)}{2j}(h(-1)+h(-2))^{2j}
  \end{align*}
\end{cor}
\begin{proof}
This follows immediately from Proposition \ref{p:pk} by setting $2k=6n(p^m-1)$.
\end{proof}

Let us summarize these computations in the following result:
\begin{thm}
  \label{t:Jnm1}
  For all $n,m \geq 1$ we have
  \begin{align*}
    J_{n,m}= & \sum_{i=0}^n\binom{n}{i} \frac{ 2^{2n+i}\cdot 3^{n-i}\cdot 5^{3i}\cdot 7^{2(n-i)}}{(6n(p^m-1)+6i-1)!! \cdot (4n-4i-1)!!} A^{6i}C^{4n-4i}\\
    &\sum_{j=0}^{3n(p^m-1)}(6n(p^m-1)-2j-1)!!(-120)^{j}\binom{6n(p^m-1)}{2j}(h(-1)+h(-2))^{2j}.
  \end{align*}
\end{thm}
\begin{proof}
  The theorem follows by combining equation \eqref{eq:Jnm2} with Corollary \ref{c:Apowers}.
\end{proof}

The integrality properties of the double factorials appearing above are particularly easy to analyze if $p=2$, because in that case they are $2$-adic units.\ Therefore at this stage we will now restrict to the case $p=2$.

\subsection{Completion of the proof when $p=2$}\label{Sp=2}
Notice now that
\[
  \frac{1}{(6n(p^m-1)+6i-1)!!} =\frac{1}{(6np^m-6(n-i)-1)!!}=\frac{\prod_{j=0}^{3(n-i)-1}((3n)2^{m+1}-2j-1)}{((3n)2^{m+1}-1)!!} 
\]
The product defining the double factorial $((3n)2^{m+1}-1)!!$ contains $3n$ copies of each representative of the unit group $(\ZZ/2^{m+1}\ZZ)^\times$. Hence for $m\geq 2$ we have
\[
  ((3n)2^{m+1}-1)!! \equiv 1 \pmod{2^{m+1}}.
\]
By combining these observations we find that for $p=2$:
\[
  \lim_{m \to \infty} \frac{1}{(6n(p^m-1)+6i-1)!!} = \prod_{j=0}^{3(n-i)-1}(-2j-1)= (-1)^{n-i}(6(n-i)-1)!!
\]
Therefore, if we write
\begin{align*}
  D_{n,m} &= \sum_{i=0}^n\binom{n}{i} \frac{ 2^{2n+i}\cdot 3^{n-i}\cdot 5^{3i}\cdot 7^{2(n-i)}}{(6n(p^m-1)+6i-1)!! \cdot (4n-4i-1)!!} A^{6i}C^{4n-4i},\\
  E_{n,m} &= \sum_{j=0}^{3n(p^m-1)}(6n(p^m-1)-2j-1)!!(-120)^{j}\binom{6n(p^m-1)}{2j}(h(-1)+h(-2))^{2j}
\end{align*}
so that $J_{n,m} = D_{n,m}(E_{n,m})$ by Theorem \ref{t:Jnm1}, we find that the following limit exists for each $n\geq 1$:
\[
  D_n \df \lim_{m \to \infty} D_{n,m} = \sum_{i=0}^n(-1)^{n-i}\binom{n}{i} \frac{ 2^{2n+i}\cdot 3^{n-i}\cdot 5^{3i}\cdot 7^{2(n-i)} \cdot (6(n-i)-1)!!}{(4(n-i)-1)!!} A^{6i}C^{4n-4i}.
\]
This limit acts on the $p$-adic Heisenberg algebra, as it is defined by a finite sum and so it can only change $2$-adic valuations by a bounded amount for each $n$.\ In fact, these are continuous operators on the $p$-adic Heisenberg algebra, and so if we can likewise show that $\lim_{m\to \infty} E_{n,m}$ exists, then we will deduce that $\lim_{m\to \infty} J_{n,m}$ exists.

First observe that because $p=2$:
\begin{align*}
  \binom{6n(p^m-1)}{2j} =& \frac{(6n(2^m-1))!}{(2j)!(6n(2^m-1)-2j)!}\\
  =&\frac{2^{3n(2^m-1)}(3n(2^m-1))!(6n(2^m-1)-1)!!}{2^j j! (2j-1)!!2^{3n(2^m-1)-j}(3n(2^m-1)-j)!(6n(2^m-1)-2j-1)!!}\\
  =&\binom{3n(2^m-1)}{j}\frac{(6n(2^m-1)-1)!!}{ (2j-1)!!(6n(2^m-1)-2j-1)!!}
\end{align*}
and thus
\begin{equation}
  \label{eq:Enm1}
  E_{n,m} = \sum_{j=0}^{3n(2^m-1)}(-120)^{j}\binom{3n(2^m-1)}{j}\frac{(6n(2^m-1)-1)!!}{ (2j-1)!!}(h(-1)+h(-2))^{2j}
\end{equation}
As above, we have the following $2$-adic identity
 \[
  \lim_{m\to \infty} (6n(2^m-1)-1)!!=\lim_{m\to \infty} \frac{((3n)2^{m+1}-1)!!}{\prod_{j=0}^{3n-1}((3n)2^{m+1}-2j-1)}=(-1)^n\frac{1}{(6n-1)!!}
\]
Likewise one shows by elementary means that
\[
\lim_{m\to \infty} \binom{3n(2^m-1)}{j} = \binom{-3n}{j}.
\]
In particular, being a limit of $2$-adic integers, the values $\binom{-3n}{j}$ are also $2$-adic integers. Thus for each $n$, the series $E_{n,m}$ converge to
\[
  E_n \df\lim_{m\to \infty} E_{n,m} = (-1)^n\frac{1}{(6n-1)!!}\sum_{j=0}^{\infty}\binom{-3n}{j}\frac{1}{ (2j-1)!!}(-120(h(-1)+h(-2))^2)^{j}.
\]
We have proved most of the following theorem.

\begin{thm}
  \label{t:main2adic}
  Let $p=2$, and define differential operators and series for all $n\geq 1$ as follows:
\begin{align*}
  D_n &= \sum_{i=0}^n(-1)^{n-i}\binom{n}{i} \frac{ 2^{2n+i}\cdot 3^{n-i}\cdot 5^{3i}\cdot 7^{2(n-i)} \cdot (6(n-i)-1)!!}{(4(n-i)-1)!!} A^{6i}C^{4n-4i},\\
  E_n &= (-1)^n\frac{1}{(6n-1)!!}\sum_{j=0}^{\infty}\binom{-3n}{j}\frac{1}{ (2j-1)!!}(-120(h(-1)+h(-2))^2)^{j},\\
  J_n &= D_n(E_n).
\end{align*}
Then the following properties hold:
\begin{enumerate}
\item For every $n\geq 1$, the series $J_n$ is contained in the $2$-adic Heisenberg algebra. More precisely, if $1\leq r < 2^{3/4}$, then $J_n \in S_r$.
\item The character of $J_n$ satisfies $\hat{f}(J_n) = j^{-n}$. 
\item We have $\abs{J_n} \leq 2^{21n}$ for all $n\geq 1$, where the absolute value is the $2$-adic supremum norm.
\end{enumerate}
\end{thm}
\begin{proof}
We have already explained why the first part of part (1) holds.\ For the second part of (1), notice that $D_n$ is a finite differential operator that preserves each subspace $S_r$.\ Therefore, to establish (1) it suffices to show that $E_n \in S_r$ for values of $r$ in the specified range.\ Since $p=2$ we can ignore the double factorials when analyzing the integrality properties of $E_n$.\ Then the second part of (1) follows immediately from the definition of $S_r$ thanks to the factor of $2^{3j}$ appearing in the coefficients via the factor $(-120)^j$.

Part (2) also follows from the previous discussion, so it only remains to discuss part (3).\ For this, notice that $2^4A$ and $2^4C$ are both $2$-adically integral.\ Therefore $2^{21n}D_n$ preserves $2$-adic integrality.\ Since $E_n$ is $2$-adically integral, so then is $2^{21n}J_n$ and thus $\abs{J_n}\leq 2^{21n}$ as claimed.
\end{proof}

Recall from equation (73) of \cite{Vonk} that for real numbers $r\geq 0$ the space of $r$-overconvergent $2$-adic modular forms of tame level $1$ and weight $0$ can be described as
\begin{align}\label{M0dagger}
M_0^\dagger(r) \df \left\{\sum_{n\geq 0} a_nj^{-n} \mid \abs{a_n} 2^{12nr} \to 0\right\}.
\end{align}

\begin{cor}
  \label{c:main2adic}
  When $p=2$ the overconvergent space $M_0^\dagger(7/4)$ is contained in the image of the normalized character map $\hat{f}$.
\end{cor}
\begin{proof}
This follows immediately from Theorem \ref{t:main2adic}, the continuity of the character map as established in \cite{FMpadic}, and the description of $M_{0}^\dagger(r)$.
\end{proof}

\begin{rmk}
Corollary \ref{c:main2adic} could be improved if the bound in part (3) of Theorem \ref{t:main2adic} could be improved.\ For example, surjectivity of the $2$-adic character map in weight $0$ would follow from an absolute bound on the series $J_n$, independent of $n$.\ We do not know if, or by how much, part (3) of Theorem \ref{t:main2adic} could be improved.\ A different approach could be to work with the Hauptmodul $\Delta(2\tau)/\Delta(\tau)$ on $\Gamma_0(2)$ in place of $j^{-1}$, as explained above equation (77) in \cite{Vonk}.\ Note that by \cite{Serre2}, all classical forms on $\Gamma_0(2)$ are $2$-adic modular forms of tame level one. Hence $\Delta(2\tau)/\Delta(\tau)$ is contained in Serre's ring $M_p$, and it could conceivably be in the image of the $2$-adic character map for the Heisenberg algebra. A first step in this direction would be to provide a concrete description of a state in the $2$-adic Heisenberg algebra whose character is $\Delta(2\tau)/\Delta(\tau)$.
\end{rmk}

Regarding the weights of relevant states, we have
\begin{cor}\label{corzero}
Assume $p=2$ and suppose that $2^{1/2}\leq r< 2^{3/4}$.\ Then each $J_n$ is contained in $S_r$, it has $X$-weight $0$, and it satisfies $L[0]J_n=0$.\ In particular, $S$ has an infinite-dimensional square bracket $0$-weight space.
\end{cor}
\begin{proof} 
Thanks to Theorem \ref{t:main2adic}(1), (2) we may apply Theorem \ref{thmsuwt}. Then the Corollary follows.
\end{proof}

\section{Continuous action of $S_{\alg}[\ ]$ on $S_R$}\label{Shn}
In this section we let $p$ denote an arbitrary prime. The main result is
\begin{thm}\label{thm19}
If $R\geq p^{1/p}$ then each square bracket Heisenberg mode
$h[n]$ acts continuously on $S_R$.
\end{thm}

The proof is broken into two pieces, treating the annihilating and creative modes separately.

\subsection{The operators $h[m]$ for positive $m$} In this Section we consider the operators
$h[m], m>0$.\ They are easier to handle $p$-adically than the same operators for $m<0$ just because their expressions in terms of $h(j)$ are easier do describe.\ We will prove
\begin{thm}\label{thm20} The following hold for all $m>0$:\\
(a) If $R>1$ then $h[m]$ is a bounded operator on $S_R$.\\
(b) If $R\geq p^2$ the operators $h[m]$ have uniformly bounded norms on $S_R$, indeed $|h[m]|_R\leq p^{-1}$.

\end{thm}
 It follows directly from \cite[equations (16), (17)]{KomatsuYoung} that for $m\geq0$ we have
\begin{align}\label{hformula}
    h[m+1] =  (m+1)!\sum_{j\geq m} \frac{s(j+1, m+1)}{(j+1)!}h(j+1)
\end{align}
where $s(i, m)$ is a Stirling number of the first kind (cf.\ Subsection \ref{SSspecnos}).

We treat the  case when $m=0$ separately.
\begin{lem} The operator $h[1]$ is a contraction operator on each $S_R$ for $R\geq 1$, and indeed $\abs{h[1]}_R\leq R^{-1}$.
\end{lem}
\begin{proof}
 We have $h[1]=h(1)$ so this is a special case of Lemma \ref{lemhnbnd}.   
\end{proof}

\begin{proof}[Completion of proof of Theorem \ref{thm20}] The indexing in \eqref{hformula} is chosen so as to conform to \cite[Theorem 6]{KomatsuYoung} where Komatsu-Young give the following lower bound on $\nu(s(j+1, m+1))$: given integers $j\geq m\geq 1$, let $r$ be such that $m p^r\leq j < m p^{r+1}$. Then
\[
\nu(s(j+1, m+1)) \geq \nu(j!)-\nu\left( \lfloor j/p^r \rfloor ! \right)-mr.
\]

Consequently, if $d_{j, m}$ is the coefficient of
$h(j+1)$ in equation \eqref{hformula}, then 
\[\nu(d_{j, m}) \geq \nu((m+1)!)-\nu(j+1)-\nu\left( \lfloor j/p^r\rfloor !  \right)-mr.\]

Let $t$ be the least positive integer such that
$m\leq p^t$.\ Then $\nu(j+1)\leq t+r+1$ and
\begin{align*}
\nu_p(d_{j, m}) &\geq \nu((m+1)!)-(t+r+1)-\nu\left( \lfloor j/p^r\rfloor !  \right)-mr\\
&\geq \nu(m!)-\nu((mp)!)-(t+r+1+mr)\\
&=-(t+r+1+mr+m)
\end{align*}
where the last equality follows from two applications of Legendre's formula for $\nu(n!)$.  

Thus in order to determine whether $h[m+1]$ is bounded on $S_R$ we must consider the expression
\begin{align}\label{xineq}
  \sup_{r\geq 0} p^{t+r+1+mr+m}\abs{h(j+1)}_R&\leq p^{t+r+1+mr+m}
 R^{-(j+1)}\notag \\
 & \leq \sup_{r\geq 0}  p^{mr+m+t+r+1}R^{-1-mp^r}
\end{align}
where the second inequality comes from an application of Lemma \ref{lemhnbnd}.

Suppose now that we assume that $R>1$.\ Then there is a positive integer $k$ such that $R\geq p^{1/p^k}$ and we have
\[
 p^{r(m+1)} R^{-1-mp^r}\leq p^{r(m+1)-p^{-k}-mp^{r-k}}
\rightarrow 0.
\]
and the righthand side of this inequality goes to $0$ as $r$ tends to infinity. By equation  \eqref{hformula} this shows that if $R>1$ then for fixed $m\geq 1$ indeed $h[m+1]$ is bounded.\ This completes the proof of part (a) of the Theorem.

 Turning to part (b), let us reconsider the expression \eqref{xineq}. We assert that the supremum is achieved for $r=0$ under the assumption $R\geq p^{2/(p-1)}$. To see this, for any $r\geq 1$ we have
 \begin{align*}
 &1+p+...+p^{r-1} \geq r    \\
 \Rightarrow  &\  \  2\frac{p^r-1}{p-1}\geq 2r\geq r\left(1+\frac{1}{m}\right)\\
 \Rightarrow  &\  \ R^{m(p^r-1)}\geq p^{mr+r}\\
 \Rightarrow  &\  \ \frac{p^{m+t+1}}{R^{1+m}}\geq 
 \frac{p^{mr+m+t+r+1}}{R^{1+mp^r}},
 \end{align*}
 and this proves the assertion.\ Therefore
 equation \eqref{xineq} shows that 
 \[ 
 \abs{h[m+1]}_R \leq  p^{m+t+1}R^{-1-m}
 \]
 as long as $R\geq p^{2/(p-1)}$.\ By definition of $t$ we certainly have $t\leq m$.\ So if we now assume that $R\geq p^2$ then the last displayed expression is  bounded by $p^{-1}$ for any $m$.
\end{proof}

\subsection{The operators $h[m]$ for negative $m$} 

In this Section we consider the operators $h[m]$ for  $m<0$. For $m=-1$, $-2$, $-3$ we examined continuity properties of the square-bracket modes in Section \ref{s:hminus2} via a direct arithmetic analysis. In this section we treat the general case via a different method that relies on properties of the $L[-1]$-operator. We will prove
\begin{thm}\label{thmnegop} If $R\geq p^{1/p}$ then each operator
$h[-n]$ for $n\geq 1$ is bounded on $S_R$.
    \end{thm}
\begin{proof}
To begin the proof of the Theorem, we note that
for all integers $n$ we have the identities
\begin{align}\label{commform}
 [L[-1], h[n]] &= -nh[n-1].
\end{align}
On the other hand we see immediately from equation \eqref{-1mode} and Lemma \ref{lemS}
that $L[-1]$  acts continuously on $S_R$ for each $R\geq 1$.\ It then follows from equation \eqref{commform} that if $h[-1]$ acts continuously on some $S_R$ for some $R\geq 1$, then the same is true for all $h[-n]$ for $n\geq 1$.\ As a consequence, in proving Theorem \ref{thmnegop} it suffices to prove the case $n=-1$.

We have seen that
\begin{align*}
   Y[h, z]&= \sum_{n\in\mathbf{Z}} h[n]z^{-n-1}= e^zY(h, e^z-1)  \\
  &= e^z\sum_{n\in\mathbf{Z}} h(-n-1)(e^z-1)^n
\end{align*}
and therefore as in Section \ref{s:hminus2} we deduce that
\[
 h[-1] = h(-1)+\sum_{n>1}    h(n-1)\frac{B_n^{(n)}(1)}{n!}
\]

From the definitions in Subsection \ref{SSspecnos} we have
\begin{align*}
   B_n^{\ell}(1)&= \sum_{r\geq 0} \binom{n}{r}
 B_r^{(\ell)}, &  B_r^{(\ell)}&\df B_r^{(\ell)}(0).
\end{align*}
 In particular,
\[
B_m^{(m)}(1) = \sum_{r\geq 0} \binom{m}{r}  B_r^{(m)},
\]
and therefore
\[
h[-1] = h(-1)+\sum_{m=2}^{\infty} \sum_{r=0}^m \frac{1}{r!(m-r)!} B_r^{(m)}h(m-1).
\]

Now there is a standard equality
\[
  s(m, m-r) = \binom{m-1}{m-r-1}   B^{(m)}_{r}.
\]
Thus we see that
\begin{align*}
    h[-1] &= h(-1)+\sum_{m=2}^{\infty} \sum_{r=0}^m \frac{1}{r!(m-r)!} s(m, m-r) \frac{(m-r-1)!r!}{(m-1)!}h(m-1)\\
&=h(-1)+\sum_{m=2}^{\infty} \left\{ \sum_{r=0}^m \frac{1}{(m-r)(m-1)!} s(m, m-r)\right\} h(m-1)\\
&=h(-1)+\sum_{j=1}^{\infty} \left\{ \sum_{r=0}^{j+1} \frac{1}{(j+1-r)j!} s(j+1, j+1-r)\right\} h(j)\\
&=h(-1)+\sum_{j=1}^{\infty} 
\left\{ \sum_{m=-1}^{j} \frac{1}{(m+1)j!} s(j+1, m+1)\right\} h(j).
\end{align*}

We must compute
\[
    \sup_{j, m}\abs{\left\{\frac{1}{(m+1)j!} s(j+1, m+1)\right\} h(j) }.
\]
Now we've already seen that $\nu(s(j+1, m+1))\geq
\nu(j!)-\nu(\lfloor j/p^r\rfloor) -mr$ 
where $mp^r\leq j < mp^{r+1}$.\ Therefore we must consider
\[
 \sup_{j, m} p^{mr -\nu(m+1)+
 \nu(\lfloor j/p^r\rfloor) }\abs{h(j)}_R \leq \sup_{j, m} p^{mr + \nu(\lfloor j/p^r\rfloor) }\abs{h(j)}_R.
\]
If $ m\leq p^t$ then $\lfloor j/p^r\rfloor < mp\leq p^{t+1}$, so that $p^{\nu(\lfloor j/p^r\rfloor) }\leq p^{t}$

Arguing as before we will get the boundedness of $h[-1]$ just as long as $p^{mr+t}R^{-j}$ converges to $0$ as $j$ goes to infinity. But $R\geq p^{1/p}$, so that 
\[
 p^{mr+t}R^{-j}\leq p^{mr+t-j/p}  \leq
 p^{mr+t-mp^{r-1}}
\]
and the righthand side of this inequality goes to $0$ as $r$ goes to infinity. Now the required limit follows, and Theorem 
\ref{thmnegop} is proved. \end{proof}

\bibliography{refs}{}
\bibliographystyle{plain}
\end{document}

%% file: preamble.tex
\newcommand{\cG}{\mathcal{G}}
\newcommand{\cH}{\mathcal{H}}

\newcommand{\veps}{\varepsilon}

\DeclareMathOperator{\Tr}{Tr}

\DeclareMathOperator{\Hom}{Hom}

\DeclareMathOperator{\im}{im}

\DeclareMathOperator{\Res}{Res}

\newcommand*{\df}{\mathrel{\vcenter{\baselineskip0.5ex \lineskiplimit0pt
                     \hbox{\scriptsize.}\hbox{\scriptsize.}}} =}

\DeclarePairedDelimiter\floor{\lfloor}{\rfloor}

\providecommand{\abs}[1]{\left\lvert#1\right\rvert}

\providecommand{\pseries}[2]{#1[\![ #2 ]\!]}

\newcommand{\QQ}{\mathbf{Q}}
\newcommand{\FF}{\mathbf{F}}

\newcommand{\ZZ}{\mathbf{Z}}

\newcommand{\RR}{\mathbf{R}}